\documentclass{amsart}                 % final size

\usepackage{amssymb,amsmath, amsfonts}
\usepackage{hyperref}
\usepackage[usenames]{color}
\usepackage{latexsym}
\newtheorem{theorem}{Theorem}[section]
\newtheorem{lemma}[theorem]{Lemma}
\newtheorem{corollary}[theorem]{Corollary}

\numberwithin{equation}{section}

\newcount\quantno
\everydisplay{\quantno=0}\everycr{\quantno=0}
\newcommand\quant{\advance\quantno by1
                      \ifnum\quantno=1\qquad\else\quad\fi\forall }

\newcommand\rest[1]{\kern-.1em
          \lower.5ex\hbox{$\scriptstyle #1$}\kern.05em}

\newcommand\set[1]{{\left\{#1\right\}}}

\renewcommand\mod[1]{\left\vert{#1}\right\vert}

\newcommand\norm[2]{{\Vert{#1}\Vert_{#2}}}

\newcommand\bignorm[2]{{\bigl\Vert{#1}\bigr\Vert_{#2}}}

\newcommand\wrt{\,\mathrm{d}}

\newcommand\1{{\bf 1}}

\newcommand\BC{\mathbb{C}}

\newcommand\BN{\mathbb{N}}
\newcommand\BR{\mathbb{R}}

\newcommand\BZ{\mathbb{Z}}

\newcommand\cB{\mathcal{B}}

\newcommand\cL{\mathcal{L}}   

\newcommand\cM{\mathcal{M}}

\newcommand\mlo {\cM_{\textnormal{loc}}}

\newcommand\ga{\gamma}

\newcommand\la{\lambda}

\newcommand\funnyk{k\hbox to 0pt{\hss\phantom{g}}}

\newcommand\lu[1]{L^1(#1)}

\newcommand\huat[1]{H^1(#1)}

\newcommand\whH{\widehat{\phantom{G}}\hbox to 0pt{\hss $H$}}

\newcommand\emspace{\hbox to 6pt{\hss}}

\newcommand\rmi{\hbox{\rm (i)}}
\newcommand\rmii{\hbox{\rm (ii)}}
\newcommand\rmiii{\hbox{\rm (iii)}}

\newcommand\One{{\mathbf{1}}}
\newcommand\OU{{Ornstein--Uhlenbeck}\ }

\newcommand\e{\mathrm{e}}

\newcommand\supp{\mathrm{supp}\,}

\begin{document}

\title[]{A maximal function characterisation \\ of the Hardy space \\ for 
the Gauss measure.}



\thanks{Work partially supported by
PRIN 2009 ``Analisi Armonica''.}

\author[G. Mauceri, S. Meda and P. Sj\"ogren]
{Giancarlo Mauceri, Stefano Meda and Peter Sj\"ogren}

\address{G. Mauceri: Dipartimento di Matematica \\
Universit\`a di Genova \\ via Dodecaneso 35, 16146 Genova \\ Italia\\ {\it mauceri@dima.unige.it}}
\address{S. Meda: Dipartimento di Matematica e Applicazioni
\\ Universit\`a di Milano-Bicocca\\
via R.~Cozzi 53\\ 20125 Milano\\ Italy\\ {\it stefano.meda@unimib.it}}
\address{P. Sj\"ogren:  Mathematical Sciences\\ 
University of Gothenburg\\ and  Mathematical Sciences\\Chalmers \\
 SE-412 96 G\"oteborg \\ Sweden\\ {\it peters@chalmers.se}}

\begin{abstract}
An atomic Hardy space  $\huat{\gamma}$ associated to the Gauss measure $\ga$ 
in $\BR^n$ has been introduced by the first two
authors. We first prove that it is equivalent to use $(1,r)$- or $(1,\infty)$-atoms
to define this $\huat{\gamma}$. For $n=1$, a maximal function characterisation
 of $\huat{\gamma}$ is found. In arbitrary dimension, we give  a description of the 
 nonnegative
functions in $\huat{\gamma}$ and  use it to prove 
that $L^p(\gamma)\subset H^1(\gamma)$ for $1<p\le\infty$. 
\end{abstract}
\maketitle

\section{Introduction}
Denote by $\ga$ the Gauss measure on $\BR^n$, i.e., the probability
measure with density $\gamma_0(x)=\pi^{-n/2}\, \e^{-\mod{x}^2}$ 
with respect to the Lebesgue measure $\lambda$. Harmonic analysis 
on the measured metric space $(\BR^n,d,\ga)$, where 
$d$ denotes the Euclidean distance on $\BR^n$, has been the object
of many investigations.  In particular, efforts have been made
to study operators related to the \OU\ semigroup, with emphasis
on maximal operators \cite{S,GU,MPS1, GMMST2, MvNP}, 
Riesz transforms \cite{Mu, Gun, M1, Pisier, Pe, Gut, GST, FGS, FoS, 
GMST1, PS, U, DV, MM}, functional calculus \cite{GMST2,GMMST1,HMM,MMS} 
and, recently, tent spaces \cite{MvNP2}.
\par
In \cite{MM} the first two authors defined an atomic Hardy-type space 
$\huat{\ga}$ and a space $BMO(\gamma)$ of functions of bounded 
mean oscillation, associated to~$\ga$.  We briefly recall their 
definitions.  A closed Euclidean ball $B$ is called \emph{admissible} at scale 
$s>0$ if 
$$ 
r_B \leq s\ \min\bigl(1,1/\mod{c_B}\bigr);
$$
here and in the sequel $r_B$ and $c_B$ denote the radius and the centre of $B$, 
respectively. We  denote by $\cB_s$ the family of all balls 
admissible at scale $s$. For the sake of brevity, we shall refer to balls 
in $\cB_1$ simply as admissible balls. Further, $B$ will be called 
\emph{maximal admissible} if $r_B =  \min(1,1/\mod{c_B})$.

Now let $r\in(1,\infty]$. A  \emph{Gaussian} $(1,r)$-\emph{atom}  is either 
the constant function $1$ or 
a function $a$ in $L^r(\gamma)$
supported in an admissible ball $B$ and such that
 \begin{equation}\label{f: prop atom}
\int a \, \wrt \ga =0  \qquad \textnormal{and}\qquad \norm{a}{r}\leq 
\ga(B)^{1/r-1};\
\end{equation}
here and in the whole paper, $\norm{\cdot}{r}$ denotes the norm in  $L^r(\gamma)$.
In the latter case, we say that the atom $a$ is associated to the ball $B
$.
The space $H^{1,r}(\ga)$ is then the vector space of all functions 
$f$ in $\lu{\ga}$ that admit a decomposition of the form 
$\sum_j \la_j \, a_j$ where the $a_j$ are Gaussian $(1,r)$-atoms 
and the sequence of complex numbers $\{\la_j\}$ is summable. 
The norm of $f$ in $H^{1,r}(\ga)$ is defined as the infimum
of $\sum_j\mod{\la_j}$ over all representations of $f$ as above. 

In \cite{MM} 
the spaces $H^{1,r}(\gamma)$ were defined and proved to coincide for all 
$1<r<\infty$, with equivalent norms.
In Section \ref{s:prel} we complement this by 
proving that they coincide also with the space $H^{1,\infty}(\gamma)$. 
Once this is established, we shall denote
the space by $H^1(\gamma)$ and use  the $H^{1,\infty}(\gamma)$ norm. 
Further, we shall frequently write atom for $(1,\infty)$-atom.
\par

The space $BMO(\gamma)$ consists of all functions $f$ in $L^1
(\gamma)$ such that
$$
\sup_{B\in\cB_1} \frac{1}{\gamma(B)}\int_B\mod{f-f_B}\wrt \gamma<
\infty,
$$
where  $f_B$ denotes the mean value of $f$ on $B$, taken with respect to 
the Gauss measure. The norm of a function in $BMO(\gamma)$ is
$$
\norm{f}{BMO(\gamma)}=\norm{f}{1}+\sup_{B\in\cB_1} \frac{1}
{\gamma(B)}\int_B\mod{f-f_B}\wrt \gamma.
$$\par
If, in the definitions of $\huat{\gamma}$ and $BMO(\gamma)$, we 
replace the family $\cB_1$ of admissible balls at scale $1$ by $\cB_s$ 
for any fixed $s>0$, we obtain the same spaces with equivalent 
norms, see \cite{MM}. We remark that a similar $H^1-BMO$ theory for 
more general measured metric  spaces has been developed by A. Carbonaro 
and the first two authors in \cite{CMM1, CMM2, CMM3}.\par
The main motivation for introducing these two spaces was to provide 
endpoint estimates for singular integrals associated to the Ornstein-Uhlenbeck 
operator $\cL=-(1/2)\Delta+x\cdot\nabla$, a natural self-adjoint Laplacian  
on $L^2(\gamma)$. Indeed, in \cite{MM} the first 
two authors proved that the imaginary powers of $\cL$ are 
bounded from $\huat{\gamma}$ to $L^1(\gamma)$ and from  $L^\infty
(\gamma)$ to $BMO(\gamma)$ and that  Riesz transforms of the form $\nabla^
\alpha\cL^{-\mod{\alpha}}$ and of any order are bounded from $L^\infty
(\gamma)$ to $BMO(\gamma)$. In a recent 
paper \cite{MMS2}, the authors proved that boundedness from  $\huat{\gamma}$ to 
 $L^1(\gamma)$ and  from  $L^\infty(\gamma)$ to $BMO(\gamma)$
 holds  for any first-order Riesz transform in   dimension one,
but not always in higher dimension.
\par
The definition of the space $\huat{\ga}$ closely resembles the atomic 
definition of the classical Hardy space $H^1(\lambda)$ on $\BR^n$ 
endowed with the Lebesgue measure $\lambda$, but there are two 
basic differences. 
First, the measured metric space $(\BR^n, d, \ga)
$ is non-doubling. Further, except for the constant atom, 
a Gaussian atoms must have ``small support'',
i.e.,  support  contained in an admissible ball. Despite these 
differences, $H^1(\gamma)$ shares many of the properties of $H^1(\lambda)$. 
In particular, the topological
dual of $\huat{\gamma}$ is isomorphic to $BMO(\gamma)$, an 
inequality of John-Nirenberg type holds for functions in
$BMO(\gamma)$ and the spaces $L^p(\gamma)$ are intermediate 
spaces between $\huat{\gamma}$ and $BMO(\gamma)$ for the real 
and the complex
interpolation methods.\par
It is well known that the classical Hardy space $H^1(\lambda)$ can be defined
in at least three different ways:  the \emph
{atomic} definition, the \emph{maximal} definition and the definition 
based on \emph{Riesz transforms} \cite{CW, St2}.
\par
As shown in \cite{MMS2}, in higher dimension the 
first-order Ornstein-Uhlenbeck Riesz transforms $\partial_j\cL^{-1/2}$ are 
unbounded from $\huat{\gamma}$ 
to $L^1(\gamma)$; here $\partial_j = \partial/\partial_{x_j}$. Thus  $\huat{\gamma}$ 
does not coincide with 
the space of all functions in $L^1(\gamma)$ such that
 $\partial_j\cL^{-1/2}f \in L^1(\gamma)$ for $j=1,\ldots,n$. \par
This paper arose from the desire to find a maximal characterisation of 
the space $\huat{\gamma}$.  We recall that the classical space $H^1(\lambda)$ 
can be characterised as the space of all functions $f$ in $L^1(\lambda)$ whose 
\emph{grand maximal function} 
\begin{equation}\label{gmf}
\cM f(x)=\sup\set{\mod{\phi_t*f(x)}: {\phi\in \Phi},\, t>0}
\end{equation}
is also in $L^1(\lambda)$. Here
 $\Phi=\set{\phi \in C_c^1\big(B(0,1)\big): \mod{D^{\alpha}\phi}\le 1 \:
\mathrm{for} \: \mod{\alpha}=0,1}$
and $\phi_t(x)=t^{-n}\phi(x/t)$.\par  
To characterise $H^1(\gamma)$, we introduce the \emph{local 
grand maximal function} defined on $L^1_{\rm loc}(\BR^n,\gamma)$ by
$$
\mlo f(x)=\sup\set{\mod{\phi_t*f(x)}: {\phi\in \Phi},\: 0<t< \min(1, 1/|x|)}.
$$
In Section \ref{maxchar1} we shall prove that, in arbitrary dimension,
$f\in \huat{\gamma}$ implies $\mlo f\in L^1(\gamma)$. Moreover, \emph
{in dimension one}  $\huat{\gamma}$ can be characterised 
as the space of all functions $f$ in $L^1(\gamma)$ satisfying 
$\mlo  f \in L^1(\gamma)$ and  the following  additional \emph{global} 
condition
\begin{equation}\label{GC}
E(f)=\int_{0}^\infty x\ \left(\mod{\int_x^\infty f\wrt\gamma}+\mod{\int_{-
\infty}^{-x} f\wrt \gamma}\right) \wrt\lambda(x)<\infty.
\end{equation}
This is Theorem \ref{char} below.
\par
Roughly speaking, if we interpret a function $f$ as a density of electrical charge 
on the real line, this global condition says  that   the positive and 
negative charges nearly balance out, so that the net charges inside the 
intervals  $(-\infty,-x)$ and $(x,\infty)$ decay sufficiently fast as $x$ 
approaches $+\infty$. The 
condition is violated when the distance between the positive and the 
negative charges increases too much or the charges do not decay 
sufficiently fast at infinity. 
For instance, let $(a_n)_1^\infty$ and $(a_n')_1^\infty$ be increasing
sequences in $(2, \infty)$ such that 
\[
a_n + 2/a_n < a_n' \quad \mathrm{and} \quad a_n' + 2/a_n' < a_{n+1} < 2a_n
\]
for all $n$. Then set
\begin{equation}
  \label{eq:ex}
  f = \sum_1^\infty c_n \left(\frac{\chi_{(a_n,a_n+1/a_n)}}{\ga(a_n,a_n+1/a_n)}
-\frac{\chi_{(a_n',a_n'+1/a_n')}}{\ga(a_n',a_n'+1/a_n')}\right)
\end{equation}
for some $c_n > 0$. One easily verifies that  $\mlo f \in L^1(\ga)$ 
if and only if $\sum c_n < \infty$. But the global condition $E(f) < \infty$
is equivalent to $\sum c_n a_n(a_n' - a_n) < \infty$, which is here  a
stronger condition.
\par\medskip

We have not been able to find a similar characterisation of $\huat
{\gamma}$ in higher dimension. However, in Section \ref{maxcharpos} 
we prove in all dimensions that if $\mlo{f} \in L^1(\gamma)$ and the 
function $f$ 
satisfies the stronger global condition
$$
E_+(f)=\int \mod{x}^2 \mod{f(x)}\wrt\gamma(x)<\infty,
$$
then $f \in \huat{\gamma}$.  Observe that for $n=1$ and $f \ge 0$,
Fubini's theorem implies that the conditions $E(f)<\infty$ and 
$E_+(f)<\infty$ are equivalent. In arbitrary dimension, $E_+(f)$  can
be used to characterise the nonnegative functions in $\huat{\gamma}$;
see Theorem \ref{pos}.
This also leads to a simple proof of the inclusions $L^p
(\gamma)\subset \huat{\gamma}$ and 
$BMO(\gamma)\subset L^{p'}(\gamma)$ for $1<p\le\infty$.

We end the introduction with some technical observations and notation.
In the following we use repeatedly the fact that on 
admissible balls at a fixed scale $s$, the Gauss and the Lebesgue 
measures are equivalent, i.e., there exists a constant $C(s)$ such that 
for every measurable subset $E$ of $B\in\cB_s$
\begin{equation}\label{lgeq}
C(s)^{-1} \gamma(E)\le \,\gamma_0(c_B) \lambda(E)\le\,C(s)\gamma
(E).
\end{equation}
In particular this implies that the Gauss measure is doubling on balls 
in $\cB_s$, with a  constant that depends on $s$ (see \cite
[{Prop. 2.1}]{MM}). Further, it is straightforward to see that if $B'\subset B$
are two balls and $B\in \cB_s$ then $B'$ is also in $\cB_s$.
\par
Given a ball $B$ in $\BR^n$ and a positive number $\rho$, we shall 
denote by $\rho B$ the ball with the same centre and with radius $\rho 
r_B$.

In the following  $C$  denotes a constant whose value may change from occurrence 
to occurrence and which depends only on the dimension $n$, except when otherwise 
explicitly stated.

\section{Coincidence of $H^{1,\infty}(\gamma)$ and $H^{1,r}(\gamma)$} \label{s:prel}

First we need a lemma which will play a role also in the maximal 
characterisation. It  deals with the classical Hardy space  $H^1(\lambda)$ with respect 
to the Lebesgue measure and the associated standard $(1,\infty)$-atoms, called  Lebesgue atoms below.  
The result is probably well known, but we include a proof because we have not been able 
to find a reference in the literature. Some related results can be found in  \cite{CKS} and \cite{CMS}.

\begin{lemma}\label{cptsupp}
If $g\in H^1(\lambda)$  and the support of $g$ is contained in a ball $B$, 
then $g$ has an atomic decomposition $g=\sum_k\lambda_k a_k$ 
where the $a_k$ are Lebesgue  $(1,\infty)$-atoms associated to balls
contained in $2B$ and 
\begin{equation}\label{ne}
\sum_k \mod{\lambda_k}\le C\,\norm{g}{H^1(\lambda)}.
\end{equation}
\end{lemma}
\begin{proof} 
In this proof, all atoms are Lebesgue  $(1,\infty)$-atoms.
We claim that  the grand maximal function of $g$ satisfies 
\begin{equation}\label{MgIc}
\cM g(x)\le\,\frac{C\,\norm{g}{1}}{\mod{B}}\qquad\forall x\notin 2B.
\end{equation}
To prove this, take $x \notin 2B$ and observe that
$$
\cM g(x) =\sup_{\phi\in\Phi}\, \sup_{t\ge d(x,B)}\mod{\phi_t*g(x)},
$$
since $\phi_t*g(x)=0$ for $0<t < d(x,B)$. Because of the vanishing
integral of $g$, one has for $\phi\in \Phi$ and $t \ge d(x,B)$
\begin{align*}
\mod{\phi_t*g(x)}&\le \,\int_{B}\mod{\phi_t(x-y)-\phi_t(x-c_B)}\,\mod{g
(y)} \wrt\lambda(y)
\\ 
& \le\,C\,t^{-n-1}\int_{B}\mod{y-c_{B}}\,\mod{g(y)}\wrt\lambda(y)\\ 
&\le\, \frac{C\,r_{B}}{d(x,B)^{n+1}}\, \norm{g}{1}\\
&\le\, \frac{C
 }{\mod{B}}\, \norm{g}{1}.
\end{align*}
\par
 For each integer $k$, denote by $\Omega_k$ the level set $\set{x: 
\cM g(x)>2^k}$. 
 Then (\ref{MgIc}) implies that 
\begin{equation}\label{hi}
 \Omega_k\subset 2B \quad \textrm{when}\quad  2^k \ge \frac{C\,
\norm{g}{1}\,}{ \mod{B}}.
\end{equation}
Let $\Omega_k=\bigcup_i Q^k_i$ be a Whitney decomposition of $
\Omega_k$  into closed cubes $Q^k_i, \; i \in \BN$, whose interiors are disjoint, 
and whose diameters are comparable to $\delta$ times their distances 
from $\Omega_k^c$, where $\delta$ is a (small) positive constant to 
be chosen later. Define $\tilde{Q}_i^k$ as the cube with the same centre 
as $Q_i^k$ and side length expanded by a factor $2$. Then $\bigcup_i\tilde{Q}^k_i=
\Omega_k$ and the family $\set{\tilde{Q}^k_i: i\in \BN}$ will have the 
bounded overlap property, uniformly in $k$, provided that $\delta$ is small enough. 
Proceeding as in the proof of the 
atomic decomposition for $H^1(\lambda)$ in \cite[p. 107--109]{St2}, 
one  shows that there exists a decomposition
\begin{equation}\label{ad}
g=\sum_{k,i} A^k_i,
\end{equation} 
with the following properties.
\begin{itemize}
\item[\rmi] Each function $A^k_i$ is supported in a ball $B^k_i$ that contains 
the cube $\tilde{Q}^k_i$ as well as those $\tilde{Q}^{k+1}_m$ that 
intersect $\tilde{Q}^k_i$. 
Moreover,  if $\delta$ is sufficiently small, the ball $B^k_i$ is 
contained in $\Omega_k$ and for each $k$ the family $\set{B^k_i}_i$ 
has the bounded overlap property.
\item[]
\item[\rmii]  $A^k_i=\lambda_{k,i} \,a^k_i$, where $a^k_i$  is 
a Lebesgue atom associated to the ball $B^k_i$ and 
$$\sum_{k,i} \mod{\lambda_{k,i}}\le C\, \norm{g}{H^1(\lambda)}.
$$
\item[]
\item[\rmiii] $\mod{A^k_i}\le C\,2^k$ for each $k$ and $i$.

\end{itemize}
 We split the sum in (\ref{ad}) in two parts
\begin{align*}
g&=\sum_{B^k_i\not\subset 2B}A^k_i\ +\sum_{B^k_i\subset 
2B} A^k_i
=\Sigma_1+\Sigma_2. 
\end{align*}
Clearly, $\Sigma_2$ is an atomic decomposition with atoms 
associated to  balls contained in $2B$. Thus, it suffices to show that $ 
\Sigma_1$ is a multiple of
an atom associated to $2B$. And indeed, $\Sigma_1$ is supported in $2B
$ and has integral zero, because it is the difference of $g$ and $
\Sigma_2$, both of which have these two properties.
Moreover, by \rmi\   and  \rmiii\  
$$
\mod{\sum_i A^k_i}\le C\,2^k.
$$
   Hence by (\ref{hi})
$$
\mod{\Sigma_1}\le\sum_{ 2^k \,<\,C\,\norm{g}{1}/\mod{B}}\mod
{\sum_i A^k_i}\le \frac{C\,\norm{g}{1}}{\mod{B}}\le \frac{C}{\mod
{2B}}\,\norm{g}{H^1(\lambda)}.
$$
Thus $\Sigma_1$ is a multiple of an atom associated 
to $2B$. We thus have the desired atomic decomposition of $g$, and 
the norm estimate (\ref{ne}) also follows.
\end{proof}

\begin{theorem}\label{1infty}
For every $r$ in $(1,\infty)$, the spaces  $H^{1,r}
(\gamma)$ and $H^{1,\infty}(\gamma)$ coincide, with equivalent norms.
\end{theorem}
\begin{proof}
In this proof, the constants $C$ may depend on $r$ and $n$.
Since any Gaussian $(1,\infty)$-atom  is also a Gaussian $(1,r)$-atom, 
 $H^{1,\infty}(\gamma)$ is a subspace of $H^{1,r}(\gamma)$ and 
$\norm{f}{H^{1,r}(\gamma)}\le\norm{f}{H^{1,\infty}(\gamma)}$.
 Conversely, suppose that $a$ is a Gaussian $(1,r)$-atom 
associated to the  ball $B \in \mathcal{B}_1$. Then the function 
$a\,\gamma_0$ is a multiple of a Lebesgue $(1,r)$-atom. Indeed, 
$\int a\gamma_0\wrt\lambda=\int a\wrt\gamma=0$
and, by the equivalence of the Gauss and Lebesgue measures on 
admissible balls, 
$$
\norm{a\gamma_0}{L^r(\lambda)}\le C\,\lambda(B)^{1/r-1}.
$$
Hence, $a_0\gamma$ is in $H^1(\lambda)$ with norm at most $C$. 
By Lemma \ref{cptsupp}, it has a decomposition
$$
a\,\gamma_0=\sum_j\lambda_j \alpha_j,
$$
where each  $\alpha_j$ is a Lebesgue $(1,\infty)$-atom associated to a 
ball $B_j$ contained in $2B$. Moreover  
$$
\sum_j\mod{\lambda_j}\le C,
$$
and each $B_j$ is admissible at scale $2$.
Define $a_j=\alpha_j\,\gamma_0^{-1}$. Then $\int a_j\wrt\gamma=0$, 
and by the equivalence of the Gauss and Lebesgue measures on 
$B_j$ 
\begin{align*}
\norm{a_j}{\infty}\le C \gamma(B_j)^{-1}.
\end{align*}
Thus the $a_j$ are  multiples of Gaussian $(1,\infty)$-atoms. Since $a=\sum_j\lambda_j a_j$, 
we conclude that $a \in H^{1,\infty}(\gamma)$ and
\begin{equation}\label{nah1}
\norm{a}{H^{1,\infty}(\gamma)}\le\,C \sum_j\mod{\lambda_j}\le C.
\end{equation}
\end{proof}

\section{The  characterisation of $H^1(\gamma)$ in $\BR$}\label
{maxchar1}
In this section, we shall prove that  $f\in \huat
{\gamma}$ implies $\mlo f\in L^1(\gamma)$ and that, in dimension one,  
functions in $H^1(\gamma)$ can be characterised by the two conditions 
$\mlo{f}\in L^1(\gamma)$ and  $E(f)<\infty$. We start with a simple but useful lemma
dealing with the support of the local grand maximal function. 
\par

\begin{lemma}\label{supp}
  If $f \in L^1(\ga)$ is supported in the admissible ball $B$,
then $\supp \mlo f$ is contained in the ball   $B' = B(c_B, R)$, 
where $R = 4 \min(1, 1/|c_B|).$
\end{lemma}

\begin{proof}
  Let $x \in \supp \mlo f$.   We write $\rho = |x|$ and $c = |c_B|,$
so that $B \subset B(c_B, \min(1,1/c))$. The balls $B$ and 
$B(x,  \min(1,1/\rho))$ must intersect, and so
\begin{equation}
  \label{eq:dist}
  |x-c_B| \le \min(1,1/c) + \min(1,1/\rho).
\end{equation}
To prove the lemma, it is enough to show that 
\begin{equation}
  \label{eq:3}
  \min(1,1/\rho) \le 3 \min(1,1/c),
\end{equation}
since it then follows that $x \in B'$.
Now $c - \rho \le |x-c_B|$, so that (\ref{eq:dist}) implies 
\[
c - \min(1,1/c) \le \rho + \min(1,1/\rho).
\]
Considering the cases $c \le 1$ and $c > 1$, we conclude from this that
\begin{equation}
  \label{eq:le}
  \max(1,c)  - \min(1,1/c) \le \max(1,\rho)   + \min(1,1/\rho).
\end{equation}
The function $t \mapsto t^{-1} - t, \;\; t>0$, and its inverse
are clearly decreasing. Considering the values of this function at
$t = \min(1,1/c)$ and $\min(1,1/\rho)/3$, we see that (\ref{eq:3}) 
is equivalent to 
\[
\max(1,c)  - \min(1,1/c) \le 3\max(1,\rho)   - \frac13 \min(1,1/\rho).
\]
Because of (\ref{eq:le}), this inequality follows if 
\[
\max(1,\rho)   + \min(1,1/\rho) \le 3\max(1,\rho)   - \frac13 \min(1,1/\rho)
\]
or equivalently $\frac43 \min(1,1/\rho) \le 2\max(1,\rho),$ which is
trivially true. We have proved (\ref{eq:3}) and the lemma.
\end{proof}

\begin{lemma}\label{necn} If $f$ is  in $\huat{\gamma}$, then $
\mlo f\in L^1(\gamma)$ and
\[
 \norm{\mlo f}{1}\le C\norm{ f}{H^1(\ga)}.
\]
\end{lemma}
\begin{proof}
We shall prove that for any Gaussian
atom  $a$ 
\begin{equation}\label{MaEa}
\norm{\mlo a}{1}\le C, 
\end{equation}
from which the lemma follows.
\par
Since (\ref{MaEa}) is obvious if $a$ is the constant function $1$,
we assume that $a$ is  associated to an admissible ball $B$. 
By the preceding lemma, $\supp \mlo f$ is contained in the ball denoted $B'$.
\par
The integral of $\mlo a$ over $2B$ with respect 
to $\ga$ is no larger
than $C$, since 
$\mathcal{M}_{\mathrm{loc}}a \le C \sup |a| \le C/\gamma(B)$. To estimate
$\mathcal{M}_{\mathrm{loc}}a$ at a point $x$ in the remaining set
$B'\setminus 2B$, we take $\phi \in \Phi$ and $0 < t < \min(1,1/\mod{x})$
and estimate $a*\phi_t(x)$.
We can assume that $t > d(x, B)$ so that $t > |x-c_B|/2$, since 
otherwise  $\phi_t*a(x)$ will vanish. Write
\begin{equation}
  \label{eq:max}
 \phi_t*a(x) = t^{-n} \int  \left(\phi\big(\frac{x-y}t\big) -\phi\big( \frac{x-c_B}t\big)\right)\,a(y)
\wrt y +   t^{-n} \phi\big( \frac{x-c_B}t\big) \int a(y) \wrt y.  
\end{equation}
Here the first term to the right can be estimated in a standard way by
$$
C t^{-n-1}  \int_B  |y-c_B| \, |a(y)| \wrt y \le C\  |x-c_B|^{-n-1}\  r_B\  \ga_0(c_B)^{-1}.
$$

To deal with the second term, we estimate $\int a(y)\wrt y$, knowing that
the integral of $a$ against $\ga$ vanishes. Thus
$$
\int a(y)\wrt y = \int a(y) \frac{\ga_0(c_B) - \ga_0(y)}{\ga_0(c_B)}\wrt y.
$$
The fraction appearing here is
\begin{equation}
  \label{eq:fraction}
e^{|c_B|^2}\left(e^{-|c_B|^2} - e^{-|y|^2}\right) 
= 1 - e^{(c_B-y)\cdot(c_B+y)},
\end{equation}
and the last exponent stays bounded for $y \in B$. Thus  
the modulus of the right-hand side of (\ref{eq:fraction}) is
 at most $C|c_B-y| |c_B+y| \le C r_B (1+|c_B|).$ 
Since $\int |a|\,d\ga \le 1$, this implies
\[
|\int a(y)\wrt y| \le  C r_B (1+|c_B|) \ga_0(c_B)^{-1}.
\]
 For the
last term in (\ref{eq:max}), we thus get the bound 
$C |x-c_B|^{-n} r_B (1+|c_B|) \ga_0(c_B)^{-1}$.

Putting things together, we conclude that for $x \in B'\setminus 2B$
$$
\mathcal{M}_{\mathrm{loc}}a(x) \le  C |x-c_B|^{-n-1} r_B \ga_0(c_B)^{-1}
+  C |x-c_B|^{-n} r_B (1+|c_B|) \ga_0(c_B)^{-1}.
$$
An integration   with respect to $\wrt \ga$, or 
equivalently $\ga_0 \wrt \lambda$,
then leads to
$$
\int_{B'\setminus 2B} \mathcal{M}_{\mathrm{loc}}a(x)\wrt \ga(x) 
\le C + C r_B (1+|c_B|) \log\frac{\min(1,1/|c_B|)}{r_B} \le C,
$$ 
and  (\ref{MaEa}) is proved.
\end{proof} 

\begin{theorem}\label{char} Let $n=1$, and 
suppose that $f$ is a function in $L^1(\gamma)$. Then $f$ is in
$H^1(\gamma)$ if and only if $\mlo f\in L^1(\gamma)$ 
and $E(f)<\infty$. The norms $\norm{f}{H^1(\gamma)}$ and 
$\norm{\mlo{f}}{L^1(\gamma)}+E(f)$ are equivalent.
\end{theorem}
\begin{proof}
Suppose that $f\in\huat{\gamma}$. Then $\mlo f\in  L^1(\gamma)$ by 
Lemma \ref{necn}. To prove the necessity of the condition $E(f)<\infty$,
it suffices to show that  $E(a)<C$ 
for all Gaussian atoms $a$. This is obvious for the exceptional atom $1$.
If $a$ is associated to a ball $B \in \cB_1$, it follows from the 
inequality
$$
\mod{\int_{-\infty}^{-x} a\wrt\gamma}+\mod{\int_{x}^{\infty} a\wrt
\gamma}\le \ \1_B(x).
$$

Conversely, assume that $f$ is a function in $L^1(\gamma)$ such that 
$\mlo f\in L^1(\gamma)$ and $E(f)<\infty$. We shall prove that $f\in 
H^1(\gamma)$, by constructing a Gaussian atomic decomposition $f=
\sum_j\lambda_j\,a_j$ such that $\sum_j \mod{\lambda_j}\le C \big
(\norm{\mlo f}{1}+E(f)\big)$. 
\par
Most of the following argument, up to the decomposition (\ref{aat}), works also in 
the $n$-dimensional setting. Since we shall need it in the next section, 
we carry out that part in $\BR^n$.\par
By subtracting a multiple of the exceptional atom $1$, we may without
loss of generality
assume that 
\begin{equation}\label{int0}
\int f\wrt\gamma=0.
\end{equation}
Let $\set{B_j}$ be a covering of $\BR^n$ by maximal admissible balls. We 
can choose this covering in such way that  the family $\set{\frac12 B_j}$ 
is disjoint and $\set{4 B_j}$ has bounded overlap  
\cite[Lemma 2.4]{GMST1}. Fix a smooth nonnegative partition of unity 
$\set{\eta_j}$ in $\BR^n$ such that 
$\supp \eta_j \subset B_j$ and $\eta_j=1$ on $\tfrac{1}{2}B_j$ and verifying $
\mod{\nabla\eta_j}\le C/r_{B_j}$. Thus $f=\sum_j f
\eta_j$. We now need the following lemma.
\par
\begin{lemma}\label{Lemma1}
For $g$ in $L^1_{\rm loc}(\gamma)$ and $x \in\BR^n$ one has
\begin{equation}\label{2maxf}
\mlo (g\eta_j\gamma_0)(x)\le C\ \gamma_0(c_{B_j})\ \mlo g(x)\  \1_{4 
B_j}
(x)\qquad \forall j.
\end{equation}
\end{lemma} 
\begin{proof}
Since the support of $\eta_j$ is contained in $B_j$, the support of 
$\mlo (g\eta_j\gamma_0)$ is contained in the ball $4B_j$, because of
Lemma  \ref{supp}. Moreover, for $\phi\in\Phi$ and
$x \in 4B_j$
$$
\phi_t*(g\eta_j\gamma_0)(x)=\gamma_0(c_{B_j})\ \tilde{\phi}_t*g(x),
$$
where $\tilde{\phi}(z)=\phi(z)\eta_j(x-tz)\gamma_0(x-tz)/\gamma_0(c_
{B_j})$. Thus, 
to prove (\ref{2maxf}) it suffices to show that there exists a constant 
$C$ 
such that $\tilde{\phi}\in C\Phi$ for  $x\in 4B_j$ and 
$0<t < \min(1, 1/\mod{x})$. The support of $\tilde{\phi}$ is 
contained 
in $B(0,1)$ and 
$$
\mod{\tilde{\phi}(z)}\le \frac{\gamma_0(x-tz)}{\gamma_0(c_{B_j})}\le C,
$$
because for $|z| \le 1$
\begin{align*}
\mod{x-tz-c_{B_j}}\le \mod{x-c_{B_j}}+\mod{tz}
 \le C\,{\min(1, 1/\mod{c_{B_j}})}.
\end{align*}
Similarly $\mod{\nabla\tilde{\phi}(z)}\le C$, because the gradients $
\nabla_z\eta_j
(x-tz)$ and $\nabla_z\gamma_0(x-tz)/\gamma_0(c_{B_j})$  give the 
factors $t(1+\mod{c_{B_j}})$ 
and $t\mod{x-tz}\gamma_0(x-tz)/\gamma_0(c_{B_j})$, respectively, 
both of which 
are bounded. This concludes the proof of Lemma \ref{Lemma1}.
\end{proof}
Continuing the proof of Theorem \ref{char}, we define $b_j\in\BC$ for each $j \in \BN$ by
\begin{equation}\label{mz}
\int_{-\infty}^\infty (f-b_j)\eta_j \wrt \gamma =0. 
\end{equation}
Note that since $\eta_j=1$ on $\frac12 B_j$,
\begin{equation}\label{1st}
\mod{b_j} = \left|\frac{\int{f\eta_j}\wrt\gamma}
{\int \eta_j \wrt \ga}\right|\le C
\frac{1}{\gamma(B_j)}{\int_{B_j}\mod{f}\wrt\gamma}.
\end{equation}
We now apply Lemma \ref{Lemma1} with $g=f - b_j$ and use the
subadditivity of $\mlo$ combined with (\ref{1st}), to get
\begin{align}\label{3d}
\int \mlo\big((f-b_j)\eta_j\gamma_0\big)\wrt\lambda&\le\,C\, 
\int_{4B_j} \mlo f \, \gamma_0(c_{B_j})\wrt\lambda 
\,+\,C\, \frac
{\ga(4B_j)}{\ga(B_j)}\int_{B_j} \mod{f}\wrt \gamma \nonumber \\
&\le\,C\, \int_
{4B_j}\mlo f \wrt\gamma.
\end{align}
\par
\begin{lemma}\label{Lemma2}
The function $(f-b_j)\eta_j\gamma_0$ is in $H^1(\lambda)$ 
and
\begin{equation}\label{4d}
\norm{(f-b_j)\eta_j\gamma_0}{H^1(\lambda)}\le\,\, \int_{4B_j}\mlo f \wrt
\gamma.
\end{equation}
\end{lemma}
\begin{proof}
By the maximal characterisation of the classical space $H^1(\lambda)
$, it suffices to show that 
\begin{equation}\label{5th}
\int\cM\left( (f-b_j)\eta_j\gamma_0\right)\wrt\lambda(x)\le \,C\, \int_
{4B_j}\mlo f \wrt\gamma.
\end{equation}
Because of (\ref{3d}), all that needs to be  verified is that
\begin{equation}\label{6th}
\int\sup_{\phi\in\Phi}\ \sup_{t\ge\min(1,1/\mod{x})} \mod{((f-b_j)
\eta_j\gamma_0\big)*\phi_t(x)}\wrt\lambda(x) \le \,C\, \int_{4B_j}\mlo f \wrt
\gamma.
\end{equation} 
To prove (\ref{6th}), we split 
the integral in the left-hand side into the sum
$$
\int_{4B_j}\cdots \wrt\lambda(x)+\int_{(4B_j)^c}\cdots \wrt\lambda(x).
$$
If $x\in 4B_j$, then for  $\phi\in\Phi$ and  $t\ge\min(1,1/\mod{x})$
\begin{align*}
\mod{\phi_t*\big((f-b_j)\eta_j\gamma_0\big)(x)}&\le t^{-n} \int_{B_j}
\mod{f(y)-b_j}\wrt \gamma(y)\\ 
&\le (1+\mod{x})^n\int _{B_j}\big(\mod{f(y)}+\mod{b_j}\big)\wrt\gamma
(y) \\ 
\qquad&\le C\,(1+\mod{c_{B_j}})^n\ \,\int_{B_j}
\mod{f(y)}\wrt\gamma(y),
\end{align*} 
 the last inequality because of  (\ref{1st}). Hence
\begin{align*}
\int_{4B_j} \cdots \wrt\lambda(x)\, \le \,C\, \ \mod
{4B_j}\ (1+\mod{c_{B_j}})^n \int_{B_j}\mod{f}\wrt\gamma \,
\le\, C\,\int_{B_j}\mlo f \wrt\gamma .
\end{align*}
If $x\in (4B_j)^c$, we take $\phi$ and $t$ as before and observe that
we can assume that $t> d(x, B_j)$, since otherwise the convolution
in (\ref{6th}) will vanish. In view of (\ref{mz}) and (\ref{1st}),
we then get
\begin{align*}
\mod{\phi_t*\big((f-b_j)\eta_j\gamma_0\big)(x)}&\le\int _{B_j}\mod
{\phi_t(x-y)-\phi_t(x-c_{B_j})}\ \mod{ f(y)-b_j}\,\eta_j(y) \wrt \gamma(y) 
\\ 
&\le\,C\, t^{-n-1} \int_{B_j} \mod{y-c_{B_j}}\,\mod{f(y)-b_j}\wrt \gamma
(y)\\
&\le \,C\,\frac{1}{d(x,B_j)^{n+1}} r_{B_j}\, \int_{B_j}
\mod{f(y)}\wrt\gamma(y).
\end{align*}
This implies that
$$
\int_{(4B_j)^c}\cdots \wrt\lambda(x)\le C\,\int_{B_j}\mod{f}\wrt\gamma\le C\, \int_{B_j} \mlo f \wrt 
\gamma.
$$
We have proved (\ref{6th}) and  the lemma.
\end{proof}

We can now finish the proof of Theorem \ref{char}. By Lemmata \ref{Lemma2} and \ref{cptsupp},  
each 
function $(f-b_j)\eta_j\gamma_0$ has an atomic decomposition $
\sum_k \lambda_{jk} \alpha_{jk}$ where the $\alpha_{jk}$ are Lebesgue atoms 
 with supports in $2B_j$ and
$$
\sum_k \mod{\lambda_{jk}}\le \,C\, \int_{4B_j} \mlo f \wrt\gamma.
$$
As we saw in the proof of Theorem \ref{1infty}, each  $a_{jk}=\gamma_0^{-1} \alpha_{jk}$ is a 
 multiple of a Gaussian atom, with a factor which is
 independent of $j$ and $k$. 
Thus
\begin{align}
f=\sum_j (f-b_j)\eta_j +\sum_j b_j \eta_j
= \sum_j \sum_k \lambda_{jk} \,a_{jk} +\sum_j b_j \eta_j. \label{aat} 
\end{align}
and
$$
\sum_{j,k} \mod{\lambda_{jk}}\le \sum_j \int_{4B_j}\mlo f \wrt\gamma
\le \,C\, \norm{\mlo f}{L^1(\gamma)}.
$$
To complete the proof of Theorem \ref{char}, we  need to find an 
atomic decomposition of $\sum_j b_j\eta_j$. It is here that we must 
restrict ourselves to the one-dimensional case and that the global condition $E(f)
<\infty$ plays a role. \par
Choose the intervals $I_0=(-1,1)$, $I_j=(\sqrt{j-1},
\sqrt{j+1})$ for $j\ge 1$ and $I_j=-I_{\mod{j}}$ for $j\le -1$.  The intervals $I_j$ 
have essentially the same properties as the balls $B_j$ introduced above, and we 
can use them instead of the $B_j$ to construct $\eta_j$ and $b_j$ as before. 
To decompose now $\sum_j b_j\eta_j$,
we first normalise the functions $\eta_j$, letting
\[
\tilde{\eta}_j = \frac{\eta_j}{\int \eta_j \wrt \ga}.
\]
Then
 $b_j\eta_j = \int f \eta_j \wrt \ga \; \tilde{\eta}_j$, and clearly
\[
\sum_{j\ge k}\int f \eta_j \wrt \ga = \int f \mu_k \wrt \ga, \quad k \in \BZ,
\]
where
$\mu_k(x) = \sum_{j\ge k} \eta_j(x).$ Notice that $\int f \mu_k \wrt \ga
\to 0$ as $k \to \pm \infty$, in view of (\ref{int0}). A summation by parts now yields
\begin{equation}
  \label{eq:**}
  \sum_{j\in \BZ} \int f \eta_j \wrt \ga \; \tilde{\eta}_j
=  \sum_{k\in \BZ} \int f \mu_k \wrt \ga \; (\tilde{\eta}_k - \tilde{\eta}_{k-1}).
\end{equation}
But $\tilde{\eta}_k - \tilde{\eta}_{k-1}$ is $C$ times a Gaussian
 atom, if we use
admissible balls at some scale $s>1$ in the definition of atoms. Thus
(\ref{eq:**}) is our desired atomic decomposition of  $\sum_j b_j\eta_j$,
provided we can estimate the coefficients by showing that
\begin{equation}\label{estmuj}
 \sum_{k\in \BZ}\left| \int f \mu_k \wrt \ga \right|
\le\,C\,\left(\norm{f}{1}+ E(f)\right).
\end{equation}
To this end, observe that
$$
 \int f \mu_k \wrt \ga =\int f(x)\int_{-\infty}^x\mu'_k(y)\wrt\lambda(y)\wrt\gamma(x)=\int 
\mu'_k(y)\int_y^\infty f(x)\wrt\gamma(x)\wrt\lambda(y).
$$
Since the support of $\mu'_k$ is contained in $I_k$ and
$$
\mod{\mu_k'(y)}\,\le\, \frac{C}{\mod{I_k}}\,\le\, C \,\big(1+\mod{c_{I_k}}
\big),
$$
we obtain, using also the bounded overlap of the $I_j$,
\begin{align*}
  \sum_{k\in \BZ}\left| \int f \mu_k \wrt \ga \right| &\,\le\,C \sum_k\int_{I_k}\big(1+\mod{c_{I_k}}\big)\ 
\mod{\int_y^\infty f(x)\wrt\gamma(x)}\wrt\lambda(y) \\
&\le C\int_{-\infty}^\infty (1+\mod{y}) \mod{\int_y^\infty f\wrt \ga}\wrt
\lambda(y) \\
&= C\int_{0}^\infty (1+y) \left(\mod{\int_y^\infty f\wrt \ga}+\mod{\int_{-y}
^\infty f\wrt \ga}\right)\wrt\lambda(y) \\ 
&\le\, C\big(\norm{f}{1}+ \,E(f)\big);
\end{align*}
here we used  (\ref{int0}). 
This concludes the proof of Theorem \ref{char}. 
\end{proof}

\section{A characterisation of nonnegative functions in $H^1(\gamma)$}\label{maxcharpos}
The dimension $n$ is now arbitrary. The following lemma will be needed.
\begin{lemma}\label{trch}
Let $\phi _0=\gamma\big(B(0,1)\big)^{-1}\,\One_{B(0,1)}$. If $g \in L^\infty$ is  
supported in a maximal admissible ball $B$, then
$$
\bignorm{g-\int g \wrt\gamma\,\phi _0}{H^1(\gamma)}\le C\big(1+\mod{c_B}^2\big)
\gamma(B)\norm{g}{L^\infty}.
$$
\end{lemma}
\begin{proof}
We shall construct atoms whose supports form a chain connecting $B(0,1)$
to $B$. First
we define a finite sequence of maximal admissible balls
\[
\tilde{B} _0 = B(0,1),\, \tilde{B}_1, \ldots, \tilde{B}_{N },
\]
 all with centres
$c_{\tilde{B}_{j}}$ on the segment $[0, c_B]$. The absolute values
$\rho_j= |c_{\tilde{B}_{j}}|$    shall be increasing in $j$, and
the boundary $\partial \tilde{B}_j$ shall contain $c_{\tilde{B}_{j-1}}$
for $j = 1,\ldots,N-1$, which means that
\begin{align} \label{recursive}
\rho_{j} - \frac{1}{\rho_{j}}=\rho_{j-1}, \qquad j=1,
\ldots,N-1,
\end{align}
and $\rho_{0} = 0, \; \;\rho_{1} = 1$.
Finally, $N$ is defined so that $\tilde{B}_{N-1}$ is the first ball
of the sequence that contains $c_B$, and $\tilde{B}_{N } = B.$
Squaring (\ref{recursive}), we get
$$
\rho_{j}^2-\rho_{j-1}^2=2 - \frac1{\rho_{j}^2}\ge1,
$$
 so that $\rho_{N-1}^2 \ge N-1$. It follows that
 \begin{equation}
  \label{balls}
N \le |c_B|^2 + 1.
 \end{equation}
\par

Next, we denote by $B _j, \;\;j=1,\ldots, N$, the largest ball 
contained in $\tilde{B}_j \cap\tilde{B}_{j-1} $. Notice that the
three balls  $\tilde{B}_j, \tilde{B}_{j-1}$  and  $B _j$ have comparable
radii and comparable Gaussian measures.
Define now functions $\phi  _j$ and $g  _j$ by 
setting
\begin{align*}
\phi  _j&=\gamma(B _j )^{-1} \One_{B _j }  \quad j=1,\ldots,N, 
\\ 
g _j &=   \int  g \wrt\gamma\, \left(\phi_{j} -\phi _{j-1} \right), \quad 
j=1,\ldots,N, \\ 
g_{N+1} &=g -\int g  \wrt \gamma \: \phi_{N}.
\end{align*}
Clearly,
\begin{equation}\label{sumgj}
g-\int g\wrt\gamma\: \phi_0 = \sum_{j=1}^{N+1} g  _j.
\end{equation}
Each function $g  _j$ is a multiple of an atom. Indeed, its integral against
 $\gamma$ vanishes. Moreover,
if $1 \le j \le N$, the support of $g  _j$ is contained in $\tilde{B} _{j-1}$ 
and   
\begin{align*}
\norm{g _j}{\infty}\le (\gamma(B_{j})^{-1} + \gamma(B_{j-1})^{-1}) \int \mod{g}\wrt\gamma\,
\le \,C\, \gamma(\tilde{B}_{j-1})^{-1}\ \gamma (B)\, \norm{g}{L^\infty}.
\end{align*}
The support of $\phi_{N+1}$ is contained in $B$ and
\begin{align*}
\norm{g_{N+1}}{\infty}\le \norm{g}{\infty}+\gamma(B)^
{-1}\int \mod{g}\wrt\gamma 
\le\,C\, \norm{g}{L^\infty}.
\end{align*}
Thus
$$
\norm{g _j}{H^1(\gamma)}\le C\,\gamma(B)\, \norm{g}{L^\infty}, \qquad j=1,\ldots,N+1.
$$
Summing the coefficients in the atomic decomposition (\ref{sumgj}), we then obtain
via (\ref{balls})
\begin{align*}
\norm{g-\int g\wrt\gamma\,\phi_0}{\huat{\gamma}}
\,\le\, C\, (N+1) \,\gamma(B) \norm{g}{L^\infty} \,
\le\,C (1+\mod{c_B}^2)\,\gamma(B) \,\norm{g}{L^\infty}.
\end{align*}
The proof of the lemma is complete.
\end{proof}

\begin{theorem}\label{pos}
Suppose that $f$ is a function in $L^1(\gamma)$.
If $\mlo f$ is in $L^1(\gamma)$ and
\begin{equation}\label{nec}
E_+(f)=\int \mod{x}^2 \mod{f(x)}\wrt\gamma(x) <\infty,
\end{equation}
then $f$  is in in $\huat{\gamma}$ and 
\[
\norm{f}{\huat{\gamma}} \le C \norm{\mlo f}{1} + C E_+(f).  
\]
 If $f$ is nonnegative,  the conditions $\mlo{f}\in L^1(\gamma)$ and 
$E_+(f)<\infty$ are also necessary for $f$ to be in $\huat{\gamma}$.
\end{theorem}
\begin{proof} 
Let $f$ be a  function in $L^1(\gamma)$ such that $\mlo f\in 
L^1(\gamma)$ and  $E_+(f)<~\infty$. Write $f=c(f)+f_0$, where $c(f)=\int 
f\wrt\gamma$. Since $c(f)$ is a multiple of the exceptional atom, it 
suffices to find an atomic decomposition of $f_0$. Note that $f_0$ satisfies
$$
\mlo f_0 \in L^1(\gamma) \qquad\textrm{and}\qquad \int\mod{x}^2\, 
\mod{f_0(x)}\wrt\gamma(x)<\infty.
$$\par
Let $\set{B_j}$ be the  covering of $\BR^n$ by maximal admissible 
balls and $\set{\eta_j}$ the corresponding partition of 
unity introduced in the proof of Theorem \ref{char}.  As there, we
choose numbers $b_j\in\BC$ such that
$$
\int_{-\infty}^\infty (f_0-b_j)\eta_j \wrt \gamma =0 \qquad\forall j.
$$
Then the argument leading to (\ref{aat}) shows that
$$ f_0=\sum_j \sum_k \lambda_{jk} \,a_{jk} +\sum_j b_j \eta_j,
$$
where the $a_{jk}$ are  Gaussian atoms supported 
in $4B_j$ and 
$$
\sum_{j,k} \mod{\lambda_{jk}}\le\,C\, \norm{\mlo f_0}{L^1(\gamma)}.
$$
It remains only to prove that $\sum_j b_j \eta_j$ is in $\huat
{\gamma}$. We  write $g_j=b_j
\eta_j$ and observe that
\begin{equation}\label{intbeta}
\int \sum_j g_j\wrt\gamma=0
\end{equation}
because $f_0$ and the $a_{ij}$ have integrals zero. Thus
$$
\sum_j g_j=\sum_j\Big(g_j-\int g_j\wrt\gamma \,\phi_0\Big),
$$
where $\phi_0$ is as in Lemma \ref{trch}.
Since (\ref{1st}) remains valid for $f_0$, we have
\begin{equation}\label{n1gj}
\norm{g_j}{\infty}\le C\,\frac{1}{\gamma(B_j)} \int_{B_j}\mod{f_0}\wrt
\gamma.
\end{equation}
Lemma \ref{trch} thus applies to each $g_j$, and using also  
the bounded overlap 
of the  $B_j$ we conclude
\begin{align*}
\norm{\sum_j g_j}{H^1(\gamma)}\,\le \, C\ \sum_j\big(1+\mod{c_{B_j}}^2\big)\ \int_{B_j} \mod{f_0}\wrt\gamma 
\,\le \, C\,\int \big(1+\mod{x}^2\big)\, \mod{f_0}\wrt\gamma.
\end{align*}
This concludes the 
proof of the sufficiency and the norm estimate.
\par

The necessity of the condition $\mlo f\in L^1(\gamma)$   was obtained in
Lemma \ref{necn}. 

 To prove the necessity of (\ref{nec}), let   $0 \le f \in \huat{\gamma}$. 
We first observe
that the function $x\mapsto \mod{x}^2$ is in $BMO(\ga)$. Indeed, its
oscillation on any admissible ball is bounded. 
Since $BMO(\gamma)$ is a lattice, the 
functions $g_k(x)=\min\set{\mod{x}^2,k}$ are 
in $BMO(\gamma)$, uniformly for  $k \ge 1$.
  By the monotone convergence theorem and the 
duality between $\huat{\gamma}$ and $BMO(\gamma)$,
$$
\int \mod{x}^2 f(x)\wrt\gamma(x)=\lim_k \int g_k(x) f(x)\wrt\gamma(x)
\le C \norm{f}{\huat{\gamma}}.
$$
The theorem is proved.
\end{proof}

The following result is a noteworthy consequence of Theorem \ref{pos}.
\begin{corollary}\label{Lp}
For $1<p\le\infty$, one has  continuous inclusions $L^p(\gamma)
\subset\huat{\gamma}$ and $BMO(\gamma)\subset L^{p'}(\gamma)$, where
$p' = p/(p-1)$.
\end{corollary}
\begin{proof}
We claim that the operator $\mlo$ is bounded on $L^p(\gamma)$ for 
$1<p\le\infty$. Deferring momentarily the proof of this claim, we 
complete the proof of the corollary. Suppose that $f$ is in $L^p
(\gamma)$. Then $\mlo f$ is in $L^1(\gamma)$, because
$$
\norm{\mlo f}{1}\le \norm{\mlo f}{p}\, \le C \norm{f}{p}<\infty,
$$
since $\gamma(\BR^n)=1$. Moreover, $E_+(f)\le \norm{\mod{x}^2}{p'}\,
\norm{f}{p} < \infty$, by H\"older's inequality. Thus $f\in \huat{\gamma}$ by 
Theorem \ref{pos}. It also follows that the inclusion 
$L^p(\gamma)\subset \huat{\gamma}$ is 
continuous, and by duality we get the continuous inclusion  
$BMO(\gamma)\subset L^{p'}(\gamma)$.\par
It remains to prove the claim.
We shall use again the
covering $\{B_j\}$ from the proof of Theorem \ref{char}. 
First  we observe that the inequality
\begin{equation}  
\label{eq:Lp}
\norm{\mlo g}{p} \le C\norm{g}{p}
\end{equation}
holds when $\supp g \subset B_j$, with a constant $C$  independent of $j$. 
Indeed, $\mlo$ is bounded on $L^p(\la)$, and $\mlo g$ is supported in the  
ball $4B_j$, where the Gaussian
measure is essentially proportional to $d\la$.

 Given a function $f \in L^p(\ga)$, we write it 
as a sum $f = \sum f_j$ with  $\supp f_j \subset B_j$ and with the
sets $\{f_j \ne 0\}$ pairwise disjoint. We can then apply (\ref{eq:Lp})
to each $f_j$ and sum.
\end{proof}

%%%%%%%%%%%@@@@@@@@@@%%%%%%%%%
\end{document}